\documentclass[12pt,leqno]{article}
\usepackage{amsfonts,amsthm,amsmath}
\usepackage{amssymb}
\usepackage{color} 
\usepackage{here}

\theoremstyle{plain}
\newtheorem{thm}{Theorem}[section]
\newtheorem{prop}{Proposition}[section]

\newtheorem{cor}{Corollary}[section]

\theoremstyle{definition}
\newtheorem{df}{Definition}[section]
\newtheorem{rem}{Remark}[section]

\newcommand{\RR}{\mathbb{R}}

\begin{document}

\title{Optimal antipodal spherical codes in the space of spherical harmonics
% using 
%spherical $4$-designs in lower dimensions
}
\author{
Tsuyoshi Miezaki\thanks{Faculty of Education, University of the Ryukyus, Okinawa  
903-0213, Japan 
miezaki@edu.u-ryukyu.ac.jp
%Telephone: +81-98-895-8883, Fax: +81-98-895-8883 
(Corresponding author)
}
}

\date{}
\maketitle

\begin{abstract}
In a previous study, 
we presented a construction of spherical $3$-designs. 
In the current study, using this construction, 
we present new optimal antipodal spherical codes 
in the space of spherical harmonics. 
Our construction is a generalization of 
Bondarenko's work.
\end{abstract}

{\small
\noindent
{\bfseries Keywords:}
Spherical designs, Lattices, Spherical harmonics.\\ \vspace{-0.15in}

\noindent
2010 {\it Mathematics Subject Classification}. 
Primary 05B30;
Secondary 11H06.\\ \quad

%\noindent
%{\it Classification codes according to UDC}. 
%519.1
%\\ \quad
%}

\section{Introduction}
This study is a sequel to the previous study \cite{M}, 
which is inspired by \cite{B}. 
\cite{B} gives an optimal antipodal spherical $(35,240,1/7)$ code
whose vectors form a spherical $3$-design. 
Generalizing the study \cite{B}, 
\cite{M} gives a construction of spherical $3$-designs. 

In the current study, using this construction, 
we present optimal antipodal spherical codes 
in the space of spherical harmonics. 
To explain our results, 
we review \cite{B} and the concept of 
spherical codes.

First, we quote some results from \cite{B}. 
Let
\[
\Delta=\sum_{j=1}^{d+1}\frac{\partial^2}{\partial x_j^2}. 
\]
We say that a polynomial $P$ in $\RR^{d+1}$ is harmonic if 
$\Delta P=0$. 
For integer $k\geq 1$, 
the restriction of a homogeneous harmonic polynomial of degree 
$k$ to $S^{d}$ is called a spherical harmonic of degree $k$. 
We denote by $\mbox{Harm}_k(S^d)$ the vector space of 
the spherical harmonics of degree $k$. 
Note that, (see for example \cite{V})
\[
\dim\mbox{Harm}_k(S^d)=\frac{2k+d-1}{k+d-1}\binom{d+k-1}{k}.
\]

For $P,Q\in \mbox{Harm}_k(S^d)$, 
we denote by $\langle P,Q\rangle$ the usual inner product
\[
\langle P,Q\rangle:=\int_{S^d}P(x)Q(x)d\sigma(x), 
\]
where $d\sigma(x)$ is a normalized Lebesgue measure on the 
unit sphere $S^d$. 
For $x\in S^d$, 
there exists $P_x\in \mbox{Harm}_k(S^d)$ such that 
\[
\langle P_x,Q\rangle=Q(x)\ \mbox{for all }Q\in \mbox{Harm}_k(S^d). 
\]
It is known that 
\[
P_x(y)=g_{k,d}((x,y)), 
\]
where $g_{k,d}$ is a Gegenbauer polynomial. 
Let 
\[
G_x=\frac{P_x}{g_{k,d}(1)^{1/2}}. 
\]
It should be noted that 
\[
\langle G_x,G_y\rangle=\frac{g_{k,d}((x,y))}{g_{k,d}(1)}. 
\]
(For a detailed explanation of 
Gegenbauer polynomials, see \cite{V}.) 
Therefore, 
if we have a set 
$X=\{x_1,\ldots,x_N\}$ in $S^d$, 
then we obtain the set 
$G_X=\{G_{x_1},\ldots,G_{x_N}\}$ in $S^{\dim{\rm Harm}_k(S^{d})-1}$. 

Thereafter, we recall the concept of an optimal antipodal 
$(d +1,N ,a )$ code. 
\begin{df}[\cite{CS}]
An antipodal set 
$X = \{ x_1,\ldots,x_N\}$ in $S^d$ (i.e.~$X=-X$)
is called an antipodal spherical 
$( d +1,N ,a )$ code 
if $|(x_i,x_j )|\leq a$ for some $a > 0$ 
and all $x_i,x_i\in X$, $i\neq j$, are not antipodal. 

This code
is called optimal if, 
for any antipodal set 
$Y=\{y_1,\ldots,y_N\}$
on $S^d$, 
there exists $y_i,y_j\in Y, 
i\neq j$, that
are not antipodal and are such that
$|(y_i,y_j)|\geq a$. 
\end{df}

Let $X=\{x_1,\ldots,x_{120}\}$ be an arbitrary subset of $240$
normalized minimum vectors of the $E_8$ lattice such that 
no pair of antipodal vectors is present in $X$. 
Set $P_x(y)=g_{2,7}((x,y))$. 
\cite{B} showed that $G_X \cup -G_{X}$ is an optimal antipodal spherical $(35,240,1/7)$ code whose vectors form a spherical $3$-design, 
where
\[
-G_X:=\{-G_{x}\mid G_x\in G_X\}. 
\]
However, this fact is an example that extends to a more general setting as follows: 
The spherical $3$-design obtained by Bondarenko in \cite{B} 
is a special case of 
our main result, which is presented as the following theorem: 

\begin{thm}\label{thm:main}
Let $X$ be an $(d+1,N,a)$ code 
and 
\[
IP:=\{(x_i,x_j)\in X^2\mid x_j\neq -x_i, i\neq j \}%=\{\pm a_1, \pm a_2\}.
\]
%Moreover, 
We assume that for all $a\in IP$, 
\[
g_{2,d}(a)=\ell\ (\mbox{constant}) 
\]
and
$X$
satisfies the 
condition 
\begin{align}\label{eqn:1}
\frac{(2N)^2}{d(d+3)/2}-4N=\ell^2((2N)^2-4N). 
\end{align}
Then, $G_X\cup -G_X$ is an 
optimal antipodal  $(d(d+3)/2,2N,\ell)$ code and 
spherical $3$-design 
in $S^{\dim{\rm Harm}_2({S^d})-1}$. 
\end{thm}
%This construction is a generalization of 
%Bondarenko's work \cite{B}.

In the following, 
we give some examples satisfying the 
condition of Theorem \ref{thm:main}. 
We quote a table of the known sharp configurations which form 
spherical $t$-design with $t\geq 4$, 
together with the 600-cell. 
(For a detailed explanation of 
sharp configurations, see \cite{CK}.) 
\begin{table}[H]
\caption{Table of the known sharp configurations which form 
are spherical $t$-design with $t\geq 4$, 
together with the 600-cell. (from \cite{CK}).}
\begin{center}
\begin{tabular}{ccccc}
$n$ &$N$ &$M$ &Inner products &Name  \\\hline
%2& N N?1 cos(2?Îj/N) (1 ? j ? N/2) N-gon
%n& N? n 1 ?1/(N ? 1) simplex
%n& n+ 1 2 ?1/n simplex
%n& 2n 3 ?1, 0 cross polytope
3& 12& 5& $-1,\pm 1/\sqrt{5}$ &icosahedron\\
4& 120& 11& $-1,\pm 1/2, 0, (\pm 1 \pm \sqrt{5})/4$ &600-cell\\
8& 240 &7& $-1,\pm 1/2, 0$& $E_8$ roots\\
7& 56& 5 &$-1,\pm 1/3$ &kissing\\
6& 27& 4& $-1/2, 1/4$ &kissing/Schl\"afli\\
%5& 16 3 3/5, 1/5 kissing
24& 196560& 11& $-1,\pm 1/2,\pm 1/4, 0$ &Leech lattice\\
23& 4600& 7&$ -1,\pm 1/3, 0 $&kissing\\
22& 891& 5 &$-1/2,-1/8, 1/4$& kissing\\
23& 552& 5 &$-1,\pm 1/5$ &equiangular lines\\
22& 275& 4 &$-1/4, 1/6$ &kissing
%21& 162 3 ?2/7, 1/7 kissing
%22& 100 3 ?4/11, 1/11 Higman-Sims
%q q3+1
%q+1 (q + 1)(q3 + 1) 3 ?1/q, 1/q2 isotropic subspaces
%(4 if q = 2) (q a prime power)
\end{tabular}
\end{center}
\end{table}
%We denote by $\widetilde{G_X}$ the set $G_X\cup -G_X$
%defined in Theorem \ref{thm:main}. 
Some of these examples satisfy the condition of Theorem \ref{thm:main}: 
\begin{cor}\label{cor:main1}
There exists antipodal spherical codes whose vectors form a
spherical $3$-design with the following parameters: 
%\begin{table}
%\caption{Table of the $(d+1,N,a)$ codes in Theorem .}
\begin{center}
\begin{tabular}{ccc}
$(d+1,N,a)$\ {\rm code}&$|(x_i,x_j)|$&$|\langle G_{x_i},G_{x_j}\rangle|$ \\\hline
$(5,12,1/5)$&$\{1/\sqrt{5}\}$&$\{1/5\}$\\
$(9,120,(1+\sqrt{5})/6)$&$\{0,(\pm 1+\sqrt{5})/4,1/2\}$&$\{0,(\pm1+\sqrt{5})/6,1/3\}$\\
$(35,240,1/7)$\ {\rm \cite{B}}&$\{0,1/2\}$&$\{1/7\}$\\
$(27,56,1/27)$\ {\rm \cite{M}}&$\{1/3\}$&$\{1/27\}$\\
$(20,54,1/8)$\ {\rm \cite{M}}&$\{1/4,1/2\}$&$\{1/10,1/8\}$\\
$(299,196560,5/23)$\ {\rm \cite{M}}&$\{0,1/4,1/2\}$&$\{1/46,1/23,5/23\}$\\
$(275,4600,7/99)$&$\{0,1/3\}$&$\{1/22,7/99\}$\\
$(252,1782,3/14)$&$\{1/8,1/4,1/2\}$&$\{1/56,1/32,3/14\}$\\
$(275,552,1/275)$&$\{1/5\}$&$\{1/275\}$\\
$(252,275,1/54)$&$\{1/6,1/4\}$&$\{1/56,1/54\}$
\end{tabular}
\end{center}
%\end{table}
Moreover, $(5,12,1/5)$, $(27,56,1/27)$, $(35,240,1/7)$, and $(275,552,1/275)$ 
codes are optimal antipodal spherical codes 
whose vectors form a spherical $3$-design.
\end{cor}

The following corollary gives parameters satisfying the 
condition of Theorem \ref{thm:main}:
\begin{cor}\label{cor:main2}
Let $X$ be a $(d+1,N,a)$ code satisyfing the condition of 
Theorem \ref{thm:main}. 
Assume that there exists $3\leq d+1\leq 100$ and 
$1/\sqrt{m}\in IP$ for some $1\leq m\leq 200$ 
such that $g_{2,d}(1/\sqrt{m})=\ell\neq 0$. 
Then, in {\rm \cite{Mtable}}, 
we list the possible parameters 
$d$, $N$, and $\ell$. 
If a $(d+1,N,a)$ code 
with the parameter in {\rm \cite{Mtable}} exists, 
then there exists an 
optimal antipodal  $(d(d+3)/2,2N,\ell)$ code and 
spherical $3$-design 
in $S^{\dim{\rm Harm}_2({S^d})-1}$. 

We note that $IP$ must be a subset of ``Inner Product" in {\rm \cite{Mtable}}. 
\end{cor}

%Furthermore, \cite{B} showed that
%$G_X \cup -G_{X}$ is a spherical $3$-design, 
%using the special properties of the $E_8$ lattice. 
%However, this fact is an example that extends to a more general setting as foll%ows. 
%The spherical $3$-design obtained by Bondarenko in \cite{B} 
%is a special case of 
%our main result, which is presented as the following theorem. 

In section \ref{sec:pre}, 
we give a definition of spherical $t$-designs and 
a construction of shperical $3$-designs. 
In section \ref{sec:proof}, we give 
proofs of Theorem \ref{thm:main}, 
Corollary \ref{cor:main1}, 
and 
Corollary \ref{cor:main2}, 
along with a concluding remark. 
%In section \ref{sec:proof-cor}, we give  
%proofs of Corollary \ref{cor:main1} and \ref{cor:main2}. 
%In section \ref{sec:rem}, we give a remark. 

All computer calculations in this study were done with the help of 
%Magma \cite{Magma} and 
Mathematica \cite{Mathematica}.

\section{Preliminary}\label{sec:pre}
In this section, we explain the concept of spherical $t$-designs 
and give the construction of spherical $3$-designs.
%Before proving Theorem \ref{thm:main}, 
%we first review some properties of spherical designs. 
%The concept of a spherical $t$-design is 
%due to Delsarte--Goethals--Seidel \cite{DGS}. 
\begin{df}[\cite{DGS}]
For a positive integer $t$, a finite non-empty set $X$ in the unit sphere
\[
S^{d} = \{{ x} = (x_1, \ldots , x_{d+1}) \in \RR ^{d+1}\mid 
x_1^{2}+ \cdots + x_{d+1}^{2} = 1\}
\]
is called a spherical $t$-design in $S^{d}$ if the following condition is satisfied:
\[
\frac{1}{|X|}\sum_{{ x}\in X}f({ x})=\frac{1}{|S^{d}|}
\int_{S^{d}}f({ x})d\sigma ({ x}), 
\]
for all polynomials $f({ x}) = f(x_1, \ldots ,x_{d+1})$ 
of degree not exceeding $t$. Here, the right hand side involves the surface integral over the sphere 
and $|S^{d}|$, the volume of sphere $S^{d}$. 

\end{df}
The meaning of spherical $t$-designs is that the average value of the integral of any polynomial of degree up to $t$ on the sphere can be replaced by its average value over a finite set on the sphere. 

The following is an equivalent condition of the 
antipodal spherical $3$-designs: 
\begin{prop}[\cite{V}]\label{prop:V}
An antipodal set $X = \{x_1,\ldots, x_N\}$ in $S^d$ 
forms a spherical $3$-design if and only if
\[
\frac{1}{|X|^2}\sum_{x_i,x_j\in X}(x_i,x_j)^2=\frac{1}{d+1}. 
\]
We note that, for any $Y:=\{y_1,\ldots,y_N\}\in S^d$, 
the following inequality holds:
\[
\frac{1}{|Y|^2}\sum_{y_i,y_j\in Y}(y_i,y_j)^2\geq \frac{1}{d+1}. 
\]

%An antipodal set $X = \{x_1,\ldots, x_N\}$ in $S^d$ 
%forms a spherical $5$-design if and only if
%\begin{align}\label{eqn:4}
%\left\{
%\begin{array}{l}
%\displaystyle\frac{1}{|X|^2}\sum_{x_i,x_j\in X}(x_i,x_j)^2=\frac{1}{d+1},\\
%\displaystyle\frac{1}{|X|^2}\sum_{x_i,x_j\in X}(x_i,x_j)^4=\frac{3}{(d+3)(d+1)}%. 
%\end{array}
%\right.
%\end{align}

\end{prop}

The following corollary give a 
construction of spherical $3$-designs. 
%\begin{thm}\label{thm:3-designs}
%Let $X$ be a finite subset of sphere $S^d$ satisfying the 
%condition {\rm (\ref{eqn:4})}. 
%We set $P_x(y)=g_{2,d}((x,y))$. 
%Then $G_X\cup -G_X$ is a spherical $3$-design 
%in $S^{\dim{\rm Harm}_2({S^d})-1}$. 
%\end{thm}
We denote by $\widetilde{G_X}$ the set $G_X\cup -G_X$
defined in Theorem \ref{thm:main}. 
\begin{cor}[\cite{M}]\label{cor:3-designs}
\begin{enumerate}
\item 
Let $X$ be a spherical $4$-design in $S^d$. 
Then, $\widetilde{G_X}$ is a spherical $3$-design 
in $S^{\dim{\rm Harm}_2({S^d})-1}$. 

\item 
Let $X$ be a spherical $4$-design in $S^d$ 
and an antipodal set. 
Let $X'$ be an arbitrary subset of $X$ with $|X'|=|X|/2$ 
such that no pair of antipodal vectors is present in $X'$. 
Then, $\widetilde{G_{X'}}$ is a spherical $3$-design 
in $S^{\dim{\rm Harm}_2({S^d})-1}$. 
\end{enumerate}
\end{cor}

\section{Proofs of Main results
%Theorem \ref{thm:main} and Corollary \ref{cor:main1}
}\label{sec:proof}

In this section, we give the proofs of Theorem \ref{thm:main}, 
Corollary \ref{cor:main1}, 
and Corollary \ref{cor:main2}. 
\begin{proof}[Proof of Theorem \ref{thm:main}]
Let $X=\{x_1,\ldots,x_{N}\}$ be a $(d+1,N,a)$ code. 
We have the following Gegenbauer polynomial of degree $2$ on $S^d$: 
\begin{align*}
g_{2,d}(x)=\frac{d+1}{d}x^2-\frac{1}{d}. 
\end{align*}

First, we show that 
$G_X\cup -G_X$ is a spherical $3$-design. 
By Proposition \ref{prop:V}, it is enough to show that 
\[
\frac{1}{|X|^2}\sum_{x_i,x_j\in X}\langle G_{x_i},G_{x_j}\rangle^2=\frac{2}{d(d+3)}
\]
because
\[
\dim\mbox{Harm}_2(S^d)=\frac{d+3}{d+1}\binom{d+1}{2}=\frac{d(d+3)}{2} 
\]
and $G_X\cup -G_X$ is an antipodal set. 
%We remark that if $X$ is a spherical $t$-design, then 
%$X\cup -X$ is also a spherical $t$-design. 
In fact, by the equation (\ref{eqn:1})
\begin{align*}
\frac{1}{|X|^2}\sum_{x_i,x_j\in X}\langle G_{x_i},G_{x_j}\rangle^2
&=\frac{1}{(2N)^2}\sum_{x_i,x_j\in X}g_{2,d}((x_i,x_j))^2\\
&=\frac{1}{(2N)^2}(4N+\ell^2((2N)^2-4N))\\
&=\frac{2}{d(d+3)}. 
\end{align*}
Therefore, 
%if $X=\{x_1,\ldots,x_N\}$ is a spherical $4$-design, then 
%$G_X=\{G_{x_1},\ldots,G_{x_N}\}$ is a spherical $2$-design and 
$G_X \cup -G_X$ is a spherical $3$-design. 

Thereafter, we prove the optimality. 
%We take an arbitrary antipodal set of points 
For any antipodal set of points 
$Y = \{y_1,\ldots, y_{2N}\}$, 
%Subsequently, 
by the equation (\ref{eqn:1}), the inequality
\[
\frac{1}{(2N)^2}\sum_{i,j=1}^{2N}(y_i,y_j)^2\geq \frac{1}{\dim\mbox{Harm}_2(S^d)}
\]
hold and we have %implies that
\[
(y_i , y_j )^2 \geq \ell^2
\]
for some $y_i , y_j \in Y , i \neq j$ and 
$y_j\neq -y_i$.
%, that are not antipodal. 

The proof is completed. 

%This immediately proves the optimality of our construction. 

%Then 
%$
%\{G_{x_1},G_{x_2},\ldots,G_{x_\ell}, 
%-G_{x_1},-G_{x_2},\ldots,-G_{x_\ell}\}
%$
%provide an antipodal spherical $(\dim\mbox{Harm}_2(S^d), 2N, \ell')$ code.

%Let $X=\{x_1,\ldots,x_N\}$ be in $S^d$ and 
%$G_X=\{G_{x_1},\ldots,G_{x_N}\}$ be in $\mbox{Harm}_2(S^d)$. 
%By Proposition \ref{prop:V}, we have 
%\begin{align*}
%\left\{
%\begin{array}{l}
%\displaystyle\frac{1}{|X|^2}\sum_{x_i,x_j\in X}(x_i,x_j)^2=\frac{1}{d+1},\\
%\displaystyle\frac{1}{|X|^2}\sum_{x_i,x_j\in X}(x_i,x_j)^4=\frac{3}{(d+3)(d+1)}, 
%\end{array}
%\right.
%\end{align*}
%since $X$ is a spherical $4$-design. 
%We have the following Gegenbauer polynomial of degree $2$ on $S^d$: 
%\begin{align*}
%g_{2,d}(x)=\frac{d+1}{d}x^2-\frac{1}{d}. 
%\end{align*}

\end{proof}

Finally, we give the proofs of Corollary \ref{cor:main1} 
and Corollary \ref{cor:main2}. 
\begin{proof}[Proof of Corollary \ref{cor:main1}]
The first part follows from Corollary \ref{cor:3-designs}. 
We give the proof of optimality for the cases $(5,12,1/5)$. 
The other cases can be proved similarly. 

Let $X$ be a $(3,12,1/\sqrt{5})$ code. 
Let $X'$ be 
an arbitrary subset of $X$ with 
$|X'| = |X|/2$ such that no pair of 
antipodal vectors is present in $X'$. 
Therefore, $X'$ satisfies the condition of Theorem \ref{thm:main}. 

This completes the proof of Corollary \ref{cor:main1}. 
\end{proof}
\begin{proof}[Proof of Corollary \ref{cor:main2}]
Recall the condition of Theorem \ref{thm:main}:
\begin{align*}
\frac{(2N)^2}{d(d+3)/2}-4N=\ell^2((2N)^2-4N), 
\end{align*}
where $g_{2,d}(a)=\ell$ for all $a\in IP$. 
Using Mathematica, 
for $3\leq d+1\leq 100$ and 
for $1\leq m\leq 200$, 
we solve the equation:
\begin{align*}
\frac{(2N)^2}{d(d+3)/2}-4N=g_{2,d}(1/\sqrt{m})^2((2N)^2-4N) 
\end{align*}
We thereafter list the solutions $d+1,2N$, and $\ell$ in \cite{Mtable}. 
We eliminate the parameters, which do not satisfy the 
Fisher bounds \cite{DGS}. 

Moreover, we solve 
\[
g_{2,d}(x)=\ell
\]
and we list as "Inner Product" in \cite{Mtable}.

This completes the proof of Corollary \ref{cor:main2}. 
\end{proof}

%\section{Concluding Remark}\label{sec:rem}

\begin{rem}
Is there a $(d+1,N,a)$ code with the 
parameters in \cite{Mtable} except for the cases 
in Cororally \ref{cor:main2}? 
\end{rem}

%\begin{rem}
%\begin{enumerate}
%\item 
%%It is interesting to note that the analogies also hold for the cases $g>1$. 
%In the present paper, 
%we only consider the genus one ($g=1$) case. 
%For the cases with $g>1$, 
%do the analogies still hold? 

%\item 
%The group $H_1$ is an example of a finite unitary reflection group. 
%These groups are classified in \cite{ST}, 
%which gives rise to a natural question: 
%for the other unitary reflection groups, 
%do our analogies still hold? 

%\end{enumerate}
%\end{rem}

\section*{Acknowledgments}

%The authors thank Hiroshi Nozaki and Masashi Shinohara 
%for their helpful discussions and contributions to this research.
This work was supported by JSPS KAKENHI (18K03217).

%The authors thank Koji Chinen, Iwan Duursma, and Manabu Oura for their helpful discussions and contributions to this research.
%The first author is supported by JSPS KAKENHI (15K04775, 17K05164, 18K03217). 

%The authors would also like to thank the anonymous
%reviewers for their beneficial comments on an earlier version of the manuscript

\end{document}